\documentclass[11pt]{amsart}

\usepackage{amssymb,amscd,amsmath,hyperref,color,enumerate}
\usepackage[all,cmtip]{xy}
\usepackage{mathrsfs}
\usepackage{lipsum}
\usepackage{nicematrix,tikz, ifthen}
\usepackage{scalerel}
\usepackage{ascii}
\makeatletter
\newenvironment{claimproof}{\par
	\pushQED{\hfill$\diamondsuit$}%
	\normalfont \topsep6\p@\@plus6\p@\relax
	\trivlist
	\item[]\ignorespaces
}{%
	\popQED\endtrivlist\@endpefalse
}
\makeatother

\usepackage[margin=3cm]{geometry}

\newcommand{\N}{{\ensuremath{\mathbb{N}}}}

\newcommand{\Q}{{\ensuremath{\mathbb{Q}}}}
\newcommand{\C}{{\ensuremath{\mathbb{C}}}}

\newcommand{\F}{{\ensuremath{\mathbb{F}}}}

\usepackage[normalem]{ulem}

\newcommand{\stkout}[1]{\ifmmode\text{\sout{\ensuremath{#1}}}\else\sout{#1}\fi}

\DeclareMathOperator{\supp}{supp}

\DeclareMathOperator{\diag}{diag}
\DeclareMathOperator{\Tr}{Tr}
\DeclareMathOperator{\id}{id}
\DeclareMathOperator{\idem}{Idem}

\DeclareMathOperator{\spn}{span}

\newcommand{\ca}[1]{\ensuremath{\mathcal{#1}}}

\newcommand{\abs}[1]{\ensuremath{ {\left| #1 \right|} }}

\newtheorem{proposition}{Proposition}[section]
\newtheorem{lemma}[proposition]{Lemma}

\newtheorem{theorem}[proposition]{Theorem}
\newtheorem{claim}{Claim}

\theoremstyle{definition}

\newtheorem{example}[proposition]{Example}

\newtheorem{remark}[proposition]{Remark}

\numberwithin{equation}{section}

\newlength{\leftstackrelawd}
\newlength{\leftstackrelbwd}
\def\leftstackrel#1#2{\settowidth{\leftstackrelawd}%
	{${{}^{#1}}$}\settowidth{\leftstackrelbwd}{$#2$}%
	\addtolength{\leftstackrelawd}{-\leftstackrelbwd}%
	\leavevmode\ifthenelse{\lengthtest{\leftstackrelawd>0pt}}%
	{\kern-.5\leftstackrelawd}{}\mathrel{\mathop{#2}\limits^{#1}}}

\newcommand{\tripprox}{\mathrel{\setbox0\hbox{$\approx$}%
		\mbox{\makebox[0pt][l]{\raisebox{0.48\ht0}{$\approx$}}$\approx$}}}

\begin{document}
	
	\title[Multiplicative and Jordan multiplicative maps on structural matrix algebras]{Multiplicative and Jordan multiplicative maps on structural matrix algebras}
	
	\author{Ilja Gogi\'{c}, Mateo Toma\v{s}evi\'{c}}
	
	\address{I.~Gogi\'c, Department of Mathematics, Faculty of Science, University of Zagreb, Bijeni\v{c}ka 30, 10000 Zagreb, Croatia}
	\email{ilja@math.hr}
	
	\address{M.~Toma\v{s}evi\'c, Department of Mathematics, Faculty of Science, University of Zagreb, Bijeni\v{c}ka 30, 10000 Zagreb, Croatia}
	\email{mateo.tomasevic@math.hr}
	
	%\thanks{We thank...}
	

	
	\date{\today}
	
	\maketitle
	
	\begin{abstract}
		Let $M_n$ denote the algebra  of $n \times n$ complex matrices and let $\mathcal{A}\subseteq M_n$ be an arbitrary structural matrix algebra, i.e.\ a  subalgebra of $M_n$ that contains all diagonal matrices. We consider injective maps $\phi : \mathcal{A}\to M_n$ that satisfy the condition
		$$
		\phi(X \bullet Y) = \phi(X) \bullet \phi(Y), \quad \text{for all } X,Y \in \mathcal{A},
		$$ 
		where $\bullet$ is either the standard matrix multiplication $(X,Y)\mapsto XY$,  the Jordan product $(X,Y) \mapsto XY+YX$, or the normalized Jordan product $(X,Y) \mapsto \frac{1}{2}(XY+YX)$. We show that all such maps $\phi$ are automatically additive if and only if $\mathcal{A}$ does not contain a central rank-one idempotent. Moreover, in this case, we fully characterize the form of these maps. 
	\end{abstract}
	
	\section{Introduction}
	The interplay between the multiplicative and the additive structure of rings and algebras has been a topic of considerable interest among mathematicians. A classical result by Martindale \cite[Corollary]{Martindale} asserts that any bijective multiplicative map from a prime ring  containing a nontrivial idempotent onto an arbitrary ring  must be additive and, consequently, a ring isomorphism. In the context of matrix rings  $M_n(\ca{R})$ over a principle ideal domain $\ca{R}$, the structure of non-degenerate multiplicative maps $\phi : M_n(\ca{R}) \to M_n(\ca{R})$ (i.e.\ maps that are not zero on all zero-determinant matrices) was completely described by Jodeit and Lam in \cite{JodeitLam}. 
	Specifically, by \cite[Corollary]{JodeitLam}, every bijective multiplicative map $\phi : M_n(\mathcal{R}) \to M_n(\mathcal{R})$ has the form
	$$
	\phi(X)=T\omega(X)T^{-1}, \quad \text{ for all } X \in M_n(\ca{R}),
	$$
	for some invertible matrix $T \in M_n(\ca{R})$ and a ring automorphism $\omega$ of $\ca{R}$, where $\omega(X)$ denotes the matrix in $M_n(\ca{R})$ obtained by applying $\omega$ entrywise to $X$. Moreover, in \cite{Pierce}, Pierce demonstrated that the Jodeit-Lam characterization does not extend to matrix rings over arbitrary integral domains. More recently, in  \cite{Semrl2} \v{S}emrl provided a comprehensive description of the (non-degenerate) multiplicative endomorphisms of matrix rings over arbitrary division rings, as well as the structure of multiplicative bijective maps of standard operator algebras (i.e.\ subalgebras of bounded linear maps on a complex Banach space that contain all finite-rank operators) \cite{Semrl1}.
	
	\smallskip

	In addition to ring homomorphisms, another important class of transformations between rings is that of Jordan homomorphisms. Specifically, a \emph{Jordan homomorphism} between associative rings (algebras) $\ca{A}$ and $\ca{B}$ is an additive (linear) map $\phi : \ca{A} \to \ca{B}$ that satisfies the condition 
	\begin{equation}\label{eq:Jordan map}
		\phi(xy+yx) = \phi(x)\phi(y) + \phi(y)\phi(x), \quad  \text{ for all } x,y \in \ca{A}.
	\end{equation}
	In the case where the rings (algebras) are $2$-torsion-free, this condition is equivalent to the requirement that $\phi$ preservers squares, meaning that 
	\begin{equation*}
		\phi(x^2) = \phi(x)^2, \quad  \text{ for all } x \in \ca{A}.
	\end{equation*}
	A fundamental problem in Jordan theory, with a rich historical background, is to identify conditions on rings (algebras) $\ca{A}$ and $\ca{B}$ which ensure that any Jordan homomorphism $\phi: \ca{A} \to \ca{B}$ (typically under additional assumptions such as surjectivity)  is either multiplicative or antimultiplicative, or more generally, can be expressed as a suitable combination of such maps. For foundational results on this subject, we refer to the papers of Herstein, Jacobson-Rickart, and Smiley \cite{Herstein, JacobsonRickart, Smiley}. The theory of Jordan homomorphisms originates from Jordan algebras, a class of nonassociative algebras that appear in various fields, including functional analysis and the mathematical foundations of quantum mechanics. Most of the practically relevant Jordan algebras naturally arise as subalgebras of an associative real or complex algebra $\mathcal{A}$, equipped with the \emph{normalized Jordan product} given by
	\begin{equation}\label{eq:normalized Jordan product}
		x \circ y:=\frac{1}{2}(xy+yx), \quad  \text{ for all } x,y \in \ca{A}.
	\end{equation}
	In  more general settings, particularly when working with rings or algebras $\ca{A}$ over an arbitrary field $\F$, possibly of characteristic $2$, one typically considers the \emph{(standard) Jordan product}, defined by
	\begin{equation}\label{eq:Jordan product}
		x \diamond y:= xy + yx, \quad  \text{ for all } x,y \in \ca{A}.
	\end{equation}
	When the characteristic of $\F$ is not $2$, it is clear that a linear map $\phi$ between $\F$-algebras $\ca{A}$ and $\ca{B}$ is a Jordan homomorphism if and only if it preserves the normalized Jordan product, i.e.\ 
	\begin{equation}\label{eq:Jordan-norm}
		\phi(x\circ y)=\phi(x) \circ \phi(y), \quad  \text{ for all } x,y \in \ca{A}.
	\end{equation} 
	In \cite[Theorem 1]{Molnar}, Molnar characterizes the bijective solutions of the functional equation \eqref{eq:Jordan-norm}, when both $\ca{A}$ and $\ca{B}$ are standard (complex) operator algebras and $\ca{A}\not\cong \C$. A key consequence of this result is that such maps are automatically additive. The finite-dimensional version of Molnar's theorem (which, along with \cite[Corollary]{JodeitLam}, serves as the primary motivation for the present work) asserts that any bijective map $\phi : M_n(\C)\to M_n(\C)$, $n \ge 2$, that satisfies \eqref{eq:Jordan-norm} is of the form
	$$
	\phi(X)=T\omega(X)T^{-1} \quad \text{ or } \quad \phi(X)=T\omega(X)^tT^{-1}, \quad  \text{ for all } X \in M_n(\C),
	$$
	for some invertible matrix $T \in M_n(\C)$ and a ring automorphism $\omega$ of $\C$, where $(\cdot)^t$ stands for the transposition. For additional variants and generalizations of Molnar’s result, particularly those addressing the automatic additivity of bijective solutions of \eqref{eq:Jordan-norm}, we refer to \cite{Ji,JiLiu,LiXiao,Lu} and the references therein.
	
	\smallskip 
	
	The purpose of this paper is to extend both \cite[Corollary]{JodeitLam} and the finite-dimensional variant of  \cite[Theorem~1]{Molnar} to the setting of injective maps on  \emph{structural matrix algebras (SMAs)}. These are subalgebras of the matrix algebra $M_n(\mathbb{F})$ over a field $\mathbb{F}$ spanned by matrix units indexed by a quasi-order on the set $\{1, \ldots, n\}$. For convenience, we focus specifically on the case where $\F$ is the field $\C$ of complex numbers. A simple argument shows that SMAs are precisely subalgebras of $M_n(\mathbb{C})$ that contain all diagonal matrices (see \cite[Proposition 3.1]{GogicTomasevic}). SMAs were originally introduced by van Wyk in \cite{VanWyk} and, since then, they (and the closely related incidence algebras) have been  the subject of extensive study, including works such as \cite{Akkurt, Akkurt2, AkkurtBarkerWild, BeslagaDascalescu, BeslagaDascalescu2, BrusamarelloFornaroliKhrypchenko1, BrusamarelloFornaroliKhrypchenko2, Coelho, Coelho2,GarcesKhrypchenko, GarcesKhrypchenko2, GogicTomasevic, GogicTomasevic2, SlowikVanWyk,VanWyk}. Let us highlight that the description of the (algebra) automorphisms of SMAs was provided by Coelho in \cite[Theorem C]{Coelho}, while the description of Jordan embeddings (monomorphisms) between two SMAs in $M_n(\C)$ was established in our recent paper \cite{GogicTomasevic}. The main result of the current paper, presented in Theorem \ref{thm:main result}, provides a characterization of SMAs $\ca{A} \subseteq M_n(\C)$ with the property that any injective map $\phi : \ca{A} \to M_n(\C)$, preserving either the standard matrix multiplication, the Jordan product~\eqref{eq:Jordan product}, or the normalized Jordan product~\eqref{eq:normalized Jordan product}, is automatically additive. This occurs precisely when $\ca{A}$ does not contain a central rank-one idempotent. Moreover, in this case, we describe the exact form of such maps. Furthermore, Example \ref{ex:nonSMAexample} illustrates that Theorem \ref{thm:main result} cannot, in general, be extended to arbitrary unital subalgebras of $M_n(\C)$.
	The paper concludes with a discussion on possible extensions of Theorem \ref{thm:main result} to SMAs over more general fields (Remark \ref{rem:extension to fields}).

	\section{Notation and Preliminaries}\label{sec:prel}
	Let us now introduce some notation which will be used throughout the paper. First of all, for an arbitrary set $S$, by $|S|$ we denote its cardinality. 
	
	\smallskip
	
	Given a unital associative complex algebra $\ca{A}$, by $Z(\ca{A})$,  $\ca{A}^\times$ and $\idem(\ca{A})$ we denote its centre, the group of all invertible elements and the set of all idempotents in $\ca{A}$, respectively. By $\circ$ we denote the (normalized) Jordan product, defined by \eqref{eq:normalized Jordan product}. Note that $p \in \ca{A}$ is an idempotent if and only if it is a Jordan idempotent (i.e.\ satisfies $p \circ p = p$). For $p \in \idem(\ca{A})$ we denote $p^\perp:=1-p \in \idem(\ca{A})$. Further, for $p,q \in \idem(\ca{A})$ we write 
	$$
	p \le q \qquad \text{if} \qquad pq = qp = p
	$$ 
	and 
	$$
	p \perp q  \qquad \text{if} \qquad pq = qp =0.
	$$
	Obviously $\le$ constitutes a partial order on $\idem(\ca{A})$.   We have the following straightforward, yet useful lemma.
	\begin{lemma}\label{le:Jordan product calculations} For $p,q \in \idem(\ca{A})$ and an arbitrary $a \in \ca{A}$ we have:
		\begin{enumerate}[(a)]
			\item $p \circ a = 0$ if and only if $pa = ap = pap = 0$.
			\item $p \circ a = a$ if and only if $pa = ap = pap=a$.
			\item $p \perp q$ if and only if $p \circ q = 0$.
			\item $p \le q$ if and only if $p \circ q = p$.
		\end{enumerate}
	\end{lemma}
	\begin{proof}
		Clearly, (a) $\implies$ (c) and (b) $\implies$ (d), so we prove only (a) and (b).
		\begin{enumerate}[(a)]
			\item If $pa = ap = 0$, then trivially $p \circ a =0$.    Conversely,  $p \circ a = 0$ is equivalent to $pa+ap=0$.
			Multiplying this equality from the left and right by $p$ yields $pa = ap = -pap$. Hence, 
			$$0 = pa + ap = -2pap \implies pa = ap = 0 = pap.$$
			\item If $pa = ap = a$, then obviously $p \circ a = a$.
			Conversely, $p \circ a = a$ is equivalent to 
			$pa + ap= 2a$.
			Multiplying this equality from the left and right by $p$ yields $pa=ap=pap$. Therefore,
			$$a = \frac{1}{2}(pa + ap) = pap \implies pa = ap = pap = a.$$
		\end{enumerate}
	\end{proof}
	
	\smallskip
	
	Let $n \in \N$.
	\begin{itemize}
		\item[--] By $[n]$ we denote the set $\{1,\ldots, n\}$.
		\item[--] By $M_n=M_n(\C)$ we denote the algebra of $n \times n$ complex matrices and by  $\ca{D}_n$ its subalgebra consisting of all diagonal matrices. 
		\item[--] Given a matrix $X \in M_n$, by $r(X)$ and $\Tr(X)$ we denote the rank and the trace of $X$, respectively.
		\item[--] For $X,Y \in M_n$, by $X \propto Y$ we denote the fact that either $X = Y = 0$, or they are both nonzero and collinear.
		\item[--] For $i,j \in [n]$, by $E_{ij}\in M_n$  we denote the standard matrix unit with $1$ at the position $(i,j)$ and $0$ elsewhere. As any matrix $X = [X_{ij}]_{i,j=1}^n \in M_n$ can be understood as a map $[n]^2 \to \C, (i,j) \mapsto X_{ij}$, we consider its \emph{support }$\supp X$ as the set of all pairs $(i,j) \in [n]^2$ such that $X_{ij} \ne 0$. Moreover, for a set $S \subseteq [n]^2$ we say that $X$ is \emph{supported in $S$} if $\supp X \subseteq S$.
		\item[--] Given a ring endomorphism $\omega$ of $\C$, we use the same symbol $\omega$ to denote the induced ring endomorphism of $M_n$, defined by applying the function $\omega$ to each entry of the underlying matrix, i.e.\   
		$$\omega(X)=[\omega(X_{ij})]_{i,j=1}^n, \quad  \text{ for all } X=[X_{ij}]_{i,j=1}^n\in M_n.$$
		\item[--] Given a binary relation $\rho$ on $[n]$, for a fixed $i \in [n]$ by $\rho(i)$ and $\rho^{-1}(i)$ we denote its image and preimage by $\rho$, respectively, i.e.
		$$\rho(i)= \{j \in [n] : (i,j) \in \rho\}, \qquad \rho^{-1}(i)= \{j \in [n] : (j,i) \in \rho\}.$$
		We also write $\rho^\times$ for $\rho\setminus \{(1,1), \ldots , (n,n)\}$.
		\item[--] By a \emph{quasi-order} on $[n]$ we mean a reflexive and transitive binary relation on $[n]$.
	\end{itemize}
	
	\medskip
	
	Given a quasi-order $\rho$ on $[n]$ we define the unital subalgebra of $M_n$ by
	$$\ca{A}_\rho :=\{X \in M_n : \supp X \subseteq \rho\}=\spn\{E_{ij} : (i,j) \in \rho\},$$
	which we call a \emph{structural matrix algebra (SMA) defined by the quasi-order $\rho$}. As already noted, structural matrix algebras are precisely the subalgebras of $M_n$ that contain $\ca{D}_n$ (see \cite[Proposition 3.1]{GogicTomasevic}). We explicitly state  the following result from \cite{GogicTomasevic}, which will be used in the proof of our main result (Theorem \ref{thm:main result}) on a few occasions. 
	
	\begin{theorem}[{\cite[Theorem 3.4]{GogicTomasevic}}]\label{thm:inner diagonalization on SMA}
		Let $\ca{A}_\rho \subseteq M_n$ be an SMA and let $\ca{F} \subseteq \ca{A}_\rho$ be a commuting family of diagonalizable matrices. Then there exists $S \in \ca{A}_\rho^\times$ such that $S\ca{F} S^{-1} \subseteq \ca{D}_n$.
	\end{theorem}

	\smallskip
	
	Additionally, as in \cite{GogicTomasevic}, given a quasi-order $\rho$ on $[n]$, by $\tripprox_0$ we denote the associated  binary relation on $[n]$, given by
	$$i \tripprox_0 j \stackrel{\text{def}}\iff (i,j) \in \rho \text{ or }(j,i) \in \rho.$$
	Its transitive closure is denoted by $\tripprox$, which forms an equivalence relation. The corresponding quotient set $[n]/\mathop{\tripprox}$ is denoted by $\ca{Q}$. We refer to each element $C \in \ca{Q}$ as a \emph{central class} of $\ca{A}_\rho$, since by \cite[Remark~3.3]{GogicTomasevic} we have 
	\begin{equation*}
		Z(\ca{A}_\rho)= \{\diag(\lambda_1,\ldots,\lambda_n) \in \ca{D}_n : ( \text{ for all } i,j \in [n])(i \tripprox j \implies \lambda_i = \lambda_j)\}.
	\end{equation*}
	Specifically, $\dim Z(\ca{A}_\rho)=\abs{\ca{Q}}$. Further, for any subset  $S \subseteq [n]$ we define the corresponding diagonal idempotent of $\ca{A}_\rho$ by
	\begin{equation}\label{eq:P_S}
		P_S := \sum_{i \in S} E_{ii}.
	\end{equation}
	In particular, $(P_C)_{C \in \ca{Q}}$ is a mutually orthogonal family of idempotents in $Z(\ca{A}_{\rho})$ such that $\sum_{C \in \ca{Q}} P_C = I$ (consequently, $(P_C)_{C \in \ca{Q}}$ 
	is a basis for $Z(\ca{A}_\rho)$). For each $C \in \ca{Q}$ we can identify the two-sided ideal $P_C \ca{A}_\rho=P_C \ca{A}_\rho P_C$ of $\ca{A}_\rho$ with the subalgebra of $M_{|C|}$ obtained from $\ca{A}_\rho$ by deleting all rows and columns not in $C$. Then $P_C \ca{A}_\rho$ becomes a central SMA in $M_{|C|}$  (i.e.\ $Z(P_C  \ca{A}_\rho)$ consists only of the scalar multiples of the identity in $M_{|C|}$), so that $\ca{A}_\rho$ is isomorphic to the direct sum of central SMAs, i.e.\
	\begin{equation}\label{eq:central decomposition}
		\ca{A}_\rho \cong \bigoplus_{C \in \ca{Q}} P_C \ca{A}_\rho.
	\end{equation}
	We refer to this fact as the \emph{central decomposition} of $\ca{A}_{\rho}$.
	\smallskip
	
	Further, following \cite{Coelho}, given a quasi-order $\rho$ on $[n]$ we say that a map $g : \rho \to \C^\times$ is \emph{transitive} if it satisfies
	$$g(i,j)g(j,k) = g(i,k),  \quad  \text{ for all } (i,j), (j,k) \in \rho.$$
	Every transitive map $g$ clearly induces an (algebra) automorphism $g^*$ of $\ca{A}_\rho$, defined on the basis of matrix units as
	\begin{equation}\label{eq:inducedauto}
		g^*(E_{ij}) = g(i,j)E_{ij}, \quad  \text{ for all } (i,j) \in\rho.
	\end{equation}
	
	\smallskip

	Finally, given a complex algebra $\mathcal{A}$ and an additive (commutative) semigroup $\ca{S}$, we say that a map $\psi : \idem(\ca{A}) \to \ca{S}$ is \emph{orthoadditive} if 
	$$p \perp q \implies \psi(p+q) = \psi(p)+\psi(q), \quad \text{ for all } p,q \in \idem(\ca{A}).$$
	A prototypical example of an orthoadditive map is the rank function $\idem(M_n) \to \N_0$, $P \mapsto r(P)$.

	\section{Main result}\label{sec:mainresult}
	We begin this section by stating our main result.
	\begin{theorem}\label{thm:main result}
		Let $\ca{A}_\rho \subseteq M_n$ be an SMA and let $\phi : \ca{A}_\rho \to M_n$ be an arbitrary injective map which satisfies
		\begin{equation}\label{eq:bullet preserving}
			\phi(X \bullet Y) = \phi(X) \bullet \phi(Y), \quad  \text{ for all } X,Y \in \ca{A}_\rho,
		\end{equation}
		where $\bullet$ is either the standard matrix multiplication, the Jordan product $\diamond$, or the normalized Jordan product $\circ$. The following conditions are equivalent: 
		\begin{itemize}
			\item[(i)] $|C|\ge 2$ for all $C \in \ca{Q}$ (i.e.\ $\ca{A}_\rho$ does not contain a central rank-one  idempotent). 
			\item[(ii)] All such maps $\phi$ are additive.
			\item[(iii)] For any such map $\phi$, there exists $T \in M_n^\times$, a transitive map $g : \rho \to \C^\times$, and for each $C \in \ca{Q}$ a nonzero ring endomorphism $\omega_C$ of $\C$ and an assignment $\dagger_C$ which is either the identity (always the case when $\phi$ is assumed to be multiplicative) or the transposition such that
			$$\phi(X) = Tg^*\left(\sum_{C \in \ca{Q}} \omega_C(P_CX)^{\dagger_C}\right) T^{-1}, \quad  \text{ for all } X \in \ca{A}_\rho,$$
			where $(P_C)_{C \in \ca{Q}}$ is a basis of mutually orthogonal idempotents  of $Z(\ca{A}_\rho)$ (defined by \eqref{eq:P_S}). 
		\end{itemize}
	\end{theorem}
	In proving Theorem \ref{thm:main result}, we shall utilize the following auxiliary facts. 
	\begin{lemma}\label{le:idempotents sum of rank one}
		Let $\ca{A}_\rho \subseteq M_n$ be an SMA. An idempotent $P \in \idem(\ca{A_\rho})$ is of rank $r \in [n]$ if and only if there exist mutually orthogonal rank-one idempotents $Q_1, \ldots, Q_{r} \in \idem(\ca{A_\rho})$ such that $P=Q_1+ \cdots + Q_r$.  Moreover, if $S \subseteq [n]$ and $\supp P \subseteq S \times S$, then we can further achieve that $\supp Q_j \subseteq S \times S$ for all $j \in [r]$.
	\end{lemma}
	\begin{proof}
		We prove only the forward implication as the converse is immediate from the orthoadditivity of the rank. We focus on the second claim, as the first one follows by plugging in $S = [n]$.
		
		Suppose therefore that $P \in \idem(\ca{A_\rho})$ is an idempotent of rank $r \in [n]$ supported in $S \times S$ for some $S \subseteq [n]$. We have $P \perp P_S^\perp$ (where  $P_S \in \idem(\ca{A_\rho})$ is defined by \eqref{eq:P_S}, and $P_S^\perp=I-P_S$), so by Theorem \ref{thm:inner diagonalization on SMA} there exists $T \in \ca{A}_\rho^\times$ and diagonal idempotents $D,D' \in \idem(\ca{A_\rho})$ such that $$P = TDT^{-1}, \qquad P_S^{\perp} = TD'T^{-1}.$$
		Since $P \perp P_S^\perp$, it follows that $D \perp D'$. Set
		$$Q_j := TE_{jj}T^{-1}, \quad \text{ where } (j,j) \in \supp D$$
		and note that $\{Q_j :  (j,j) \in \supp D\}$ is a family of $r$ mutually orthogonal rank-one idempotents summing up to $P$. Further, each $Q_j$ is clearly orthogonal to $P_S^{\perp}$ and hence supported in $S \times S$.
	\end{proof}
	
	\begin{lemma}\label{le:contains entire CxC}
		Let $\rho$ be a quasi-order on $[n]$ and let $\ca{S} \subseteq \rho^\times$ be a nonempty subset. Suppose that for each $(i,j) \in \ca{S}$ we have
		$$(i,k) \in \ca{S},  \forall k \in (\rho^\times)(i), \quad (l,j) \in \ca{S}, \forall l \in (\rho^\times)^{-1}(j), \quad \text{ and }\quad  (j,i) \in \rho^\times \implies (j,i) \in \ca{S}.$$
		If for each $(i,j) \in \ca{S}$ we denote by $C \in \ca{Q}$ the central class which contains $i$ and $j$, then we have $\rho^\times \cap (C\times C) \subseteq \ca{S}$.
	\end{lemma}
	\begin{proof}
		Fix some $(i,j) \in \ca{S}$. We first prove that
		$$(j,l) \in \ca{S}, \forall l \in (\rho^\times)(j)\setminus \{i\} \quad \text{ and } \quad (k,i) \in \ca{S},  \forall k \in (\rho^\times)^{-1}(i)\setminus \{j\}.$$
		Indeed, from $(j,l) \in \rho^\times$ by transitivity it follows $(i,l) \in \rho^\times$ and hence $(i,l) \in  \ca{S}$. It follows $(j,l) \in\ca{S}$. The same argument works for the other case.

		Now by $C \in \ca{Q}$ denote the central class which contains $i$ and $j$. Denote
		$$T := \{k \in [n] : \exists l \in [n] \text{ such that }(k,l) \in \ca{S} \text{ or }(l,k) \in \ca{S}\}.$$
		In view of the assumption and the property just established, we have $k \in T$ implies that $\ca{S}$ contains all pairs $(r,s) \in \rho^\times$ such that $k \in \{r,s\}$ (note that the last implication in the assumption is used here). Therefore, to prove the claim, it suffices to show that $C \subseteq T$. Indeed, if $(r,s) \in \rho^\times \cap (C \times C)$, then $r,s \in C \subseteq T$. Since $\ca{S}$ contains all pairs in $\rho^\times$ with either coordinate in $T$, we conclude $(r,s) \in \ca{S}$. Let $k \in C$ be arbitrary. Since $i \in C$, by the definition of $\ca{Q}$ we have $i \tripprox k$ so there exist $m \in \N$ and $i_0, i_1,\ldots,i_m \in [n]$ such that
		$$i = i_0 \tripprox_0 i_1 \tripprox_0 \cdots \tripprox_0 i_m = k.$$
		Since $i_0 = i \in T$, we clearly have $i_1 \in T$. We continue inductively and conclude $k \in T$.
	\end{proof}

	\begin{lemma}\label{le:basic properties}
		Let $\ca{A}_\rho \subseteq M_n$ be an SMA and let $\phi : \ca{A}_\rho \to M_n$ be an injective map which satisfies \eqref{eq:bullet preserving}, where $\bullet$ is either the standard matrix multiplication or the (normalized) Jordan product $\circ$. Then the following holds true.
		\begin{enumerate}[(a)]
			\item $\phi$  preserves idempotents, i.e.\ $\phi(\idem(\ca{A_\rho}))\subseteq \idem(M_n)$
			\item For $P,Q \in \idem(\ca{A_\rho})$ we have $P \le Q \implies \phi(P) \le \phi(Q)$.
			\item For each $P \in \idem(\ca{A_\rho})$ we have $r(\phi(P)) = r(P)$. In particular $\phi(0) = 0$ and $\phi(I) = I$.
			\item For $P,Q \in \idem(\ca{A_\rho})$ we have $P \perp Q \implies \phi(P) \perp \phi(Q)$.
			\item For each $P \in \idem(\ca{A_\rho})$ we have $\phi(P^\perp) = \phi(P)^\perp$.
			\item The restriction $\phi|_{\idem(\ca{A_\rho})}: \idem(\ca{A_\rho})\to \idem(M_n)$ is orthoadditive.
			\item Suppose that $P_1,\ldots,P_r \in \idem(\ca{A_\rho})$ are mutually orthogonal and let $\lambda_1,\ldots,\lambda_r \in \C$. Then
			$$\phi\left(\sum_{j=1}^r \lambda_j P_j\right) = \sum_{j=1}^r \phi(\lambda_j P_j).$$
		\end{enumerate}
	\end{lemma}
	\begin{proof}
		\begin{enumerate}[(a)]
			\item This is clear.
			\item We have $$\phi(P) = \phi(P \bullet Q) = \phi(P) \bullet \phi(Q), \qquad \phi(P) = \phi(Q \bullet P) = \phi(Q) \bullet \phi(P),$$
			which is (by Lemma \ref{le:Jordan product calculations} if necessary) equivalent to  $\phi(P) \le \phi(Q)$. 
			\item In view of Theorem  \ref{thm:inner diagonalization on SMA}, note that for any $P,Q \in \idem(M_n)$ we have that $P \le Q$ implies $r(P) \le r(Q)$ with equality if and only if $P = Q$. By the same theorem, any $P\in\idem(\ca{A_\rho})$ is part of a strictly  $\le$-increasing chain of idempotents 
			$$0 = P_0 \lneq P_1 \lneq \cdots \lneq P_n = I$$ 
			in $\idem(\ca{A}_\rho)$. By (b) and the injectivity of $\phi$ it follows
			$$\phi(P_0) \lneq \phi(P_1) \lneq \cdots \lneq\phi(P_n)$$
			and hence 
			$$r(\phi(P_0)) < r(\phi(P_1)) < \cdots < r(\phi(P_n)).$$
			Clearly, for each $0 \le j \le n$ we have $r(\phi(P_j)) = j = r(P_j)$.
			\item We have
			$$\phi(P) \bullet \phi(Q) = \phi(P \bullet Q) = \phi(0) = 0, \qquad \phi(Q) \bullet \phi(P) = \phi(Q \bullet P) = \phi(0) = 0$$
			so (again by Lemma \ref{le:Jordan product calculations} if necessary) $\phi(P) \perp \phi(Q)$.
			\item In view of (c) and (d), we have that $\phi(P^\perp)$ is an idempotent orthogonal to $\phi(P)$ of rank $r(P^\perp)=r(\phi(P)^\perp)$. Therefore $\phi(P^\perp) \le \phi(P)^\perp$, so the equality of ranks implies $\phi(P^\perp) =\phi(P)^\perp$ (as already noted in (c)).
			
			\item Since $P\perp Q$, we have that $P+Q$ is again an idempotent and $P,Q \le P+Q$. Statements (b) and (d) imply $$\underbrace{\phi(P), \phi(Q)}_{\text{orthogonal}} \le \phi(P+Q)$$ and hence
			$$\phi(P)+\phi(Q) \le \phi(P+Q).$$
			Finally, we have
			\begin{align*}
				r(\phi(P) + \phi(Q)) &= r(\phi(P)) + r(\phi(Q)) \stackrel{(c)}= r(P) + r(Q) = r(P+Q) \\
				&\leftstackrel{(c)}= r(\phi(P+Q)),
			\end{align*}
			so equality follows.
			\item We have \begin{align*}
				\phi\left(\sum_{j=1}^r \lambda_j P_j\right) &= \phi\left(\left(\sum_{j=1}^r \lambda_j P_j\right) \bullet \left(\sum_{l=1}^r  P_l\right)\right) = \phi\left(\sum_{j=1}^r \lambda_j P_j\right) \bullet \phi\left(\sum_{l=1}^r P_l\right)\\
				&\leftstackrel{(f)}= \phi\left(\sum_{j=1}^r \lambda_j P_j\right) \bullet \left(\sum_{l=1}^r \phi(P_l)\right) = \sum_{l=1}^r \left(\phi\left(\sum_{j=1}^r \lambda_j P_j\right) \bullet \phi(P_l)\right)\\
				&= \sum_{l=1}^r \phi\left(\left(\sum_{j=1}^r \lambda_j P_j\right) \bullet P_l\right) = \sum_{l=1}^r \phi(\lambda_l P_l).
			\end{align*}
		\end{enumerate}
	\end{proof}

	\begin{proof}[Proof of Theorem \ref{thm:main result}]
		First of all, if $\ca{A}$ and $\ca{B}$ are arbitrary algebras over a field $\F$ of characteristic not $2$, note that all $\diamond$-preserving injective (respectively, surjective or bijective) maps $\ca{A}\to \ca{B}$ are automatically additive if and only if the same holds true for the corresponding $\circ$-preserving maps. Indeed, if  for example $\phi: \ca{A}\to \ca{B}$ is $\diamond$-preserving, then the map $\psi : \ca{A} \to \ca{B}$ defined by
		$$
		\psi(x):=2 \phi\left(\frac{x}{2}\right), \quad  \text{ for all } x \in \ca{A}
		$$
		is clearly $\circ$-preserving.  It therefore suffices to establish the result in the $\circ$-preserving case.
		
		\smallskip
		
		\fbox{(iii) $\implies$ (ii)} This is obvious.
		
		\smallskip
		
		\fbox{(ii) $\implies$ (i)} Suppose that (i) is not true. In the context of the central decomposition \eqref{eq:central decomposition} of SMAs, this precisely means that $\ca{A}_\rho$ contains a central summand isomorphic to $\C$. Denote by $\omega : \C \to \C$ some injective multiplicative function which is not additive (see e.g.\ \cite{Molnar}). Then one can construct a map $\ca{A}_\rho \to \ca{A}_\rho$ which acts componentwise as the map $\omega$ on all one-dimensional central summands of $\ca{A}_\rho$, and as the identity map on all other central summands. Such a map is clearly multiplicative and preserves the Jordan product, but is not additive.
		
		\smallskip

		\fbox{(i) $\implies$ (iii)} This implication is the core of the theorem and its proof will be divided into several steps. Let $\phi : \ca{A}_\rho \to M_n$ be an arbitrary injective $\circ$-preserving map. First of all, as $E_{11}, \ldots, E_{nn}$ is a mutually orthogonal family of rank-one idempotents in $\ca{A}_\rho$, by Lemma \ref{le:basic properties} the same is true for idempotents $\phi(E_{11}), \ldots, \phi(E_{nn})$ in $M_n$. Hence, without loss of generality we can therefore assume that
		$$\phi(E_{jj}) = E_{jj}, \quad  \text{ for all } j \in [n].$$
		Then, by the orthoadditivity of $\phi$ on $\idem(\ca{A_\rho})$ (Lemma \ref{le:basic properties} (f)), we also have 
		\begin{equation}\label{eq:identity on diagonal idempotents}
			\phi(P) = P, \quad  \text{ for all } P \in \ca{D}_n \cap \idem(\ca{A_\rho}).
		\end{equation}
		
		\begin{claim}\label{cl:preserves support}
			Let $S \subseteq [n]$ and suppose that a matrix $X \in \ca{A}_\rho$ satisfies $\supp X \subseteq S \times S$. Then $\phi(X)$ satisfies the same property.
		\end{claim}
		\begin{claimproof}
			Consider the diagonal idempotent $P:=P_S^\perp \in \idem(\ca{A}_\rho)$ (where  $P_S \in \idem(\ca{A_\rho})$ is defined by \eqref{eq:P_S}). Note that a matrix $X$ is supported in $S \times S$ if and only if $XP=PX=0$. Then obviously  $X \bullet P = P \bullet X = 0$, so 
			$$0 = \phi(X \bullet P) = \phi(X) \bullet \phi(P) \stackrel{\eqref{eq:identity on diagonal idempotents}}= \phi(X) \bullet P$$
			and similarly $0 = P \bullet \phi(X)$ which together (by Lemma \ref{le:Jordan product calculations} (a)) imply the claim.
		\end{claimproof}
		
		\begin{claim}\label{cl:h exists}
			For each central class $C \in \ca{Q}$, there exists a unique injective map $\omega_C : \C \to \C$ such that 
			\begin{equation}\label{eq:homogeneity}
				\phi(\lambda X) = \omega_C(\lambda) \phi(X), \quad  \text{ for all } \lambda \in \C \text{ and } X \in \ca{A}_\rho \text{ with } \supp X \subseteq C \times C.
			\end{equation}
		\end{claim}
		\begin{claimproof}
			Let $P \in \idem(\ca{A_\rho})$ be a rank-one idempotent and let $\lambda \in \C^\times$. Then
			$$\phi(\lambda P) = \phi((\lambda P) \bullet P) = \phi(\lambda P) \bullet \phi(P), \qquad \phi(\lambda P) = \phi(P \bullet (\lambda P)) = \phi(P) \bullet \phi(\lambda P)$$
			In view of Lemma \ref{le:Jordan product calculations} (b) we have $$\phi(\lambda P) = \phi(P)\phi(\lambda P)\phi(P).$$
			Since $\phi(\lambda P) \ne 0$ by injectivity, it follows that $\phi(\lambda P)$ has rank one, and shares the same image and kernel as $\phi(P)$ so we conclude $\phi(\lambda P) \propto \phi(P)$. Since $\phi(0) = 0$, it follows that there exists a map $\omega^P : \C \to \C$ such that $$\phi(\lambda P) = \omega^P(\lambda)\phi(P), \quad  \text{ for all } \lambda \in \C.$$
			Note that if $P, Q \in \idem(\ca{A_\rho})$ are two rank-one idempotents, we have
			\begin{equation}\label{eq:equal when they are not orthogonal}
				P \not\perp Q \implies \omega^P = \omega^Q. 
			\end{equation}
			Namely, supposing that $P \bullet Q \ne 0$, we have
			$$\phi(P) \bullet \phi(Q)  = \phi(P \bullet Q) \ne 0$$
			and hence for each $\lambda \in \C$ we have
			\begin{align*}
				\omega^Q(\lambda) (\phi(P) \bullet \phi(Q)) &= \phi(P) \bullet \phi(\lambda Q) = \phi(P \bullet (\lambda Q)) \\
				&= \phi((\lambda P) \bullet Q) = \phi(\lambda P) \bullet \phi(Q) \\
				&= \omega^P(\lambda) (\phi(P) \bullet \phi(Q)),
			\end{align*}
			so it follows $\omega^Q = \omega^P$.
			The next objective is to show that for each $C \in \ca{Q}$, the map $\omega^P$ is in fact the same for all rank-one idempotents $P \in \idem(\ca{A_\rho})$ supported in $C \times C$. As a first step, note that for any two distinct $i,j \in [n]$ we have
			$$i \tripprox j \implies \omega^{E_{ii}} = \omega^{E_{jj}}.$$
			Indeed, by transitivity it suffices to prove this assuming $i \tripprox_0 j$, i.e.\ when $(i,j) \in \rho$ or $(j,i)\in\rho$. For concreteness assume the former. Then the rank-one idempotent $ E_{ii} + E_{ij} \in \idem(\ca{A_\rho})$ satisfies
			$$E_{ii} \not\perp E_{ii} + E_{ij} \not\perp E_{jj} \stackrel{\eqref{eq:equal when they are not orthogonal}}\implies  \omega^{E_{ii}} = \omega^{E_{ii} + E_{ij}} =\omega^{E_{jj}}.$$
			Fix now some $C \in \ca{Q}$ and let $P \in \idem(\ca{A_\rho})$ be an arbitrary rank-one idempotent such that $\supp P \subseteq C \times C$. As $\Tr(P)=1$, we can pick some $(i,i) \in \supp P$. We have
			$$P \not\perp E_{ii} \stackrel{\eqref{eq:equal when they are not orthogonal}}\implies \omega^P = \omega^{E_{ii}}.$$
			So indeed we see that the map $P \mapsto \omega^P$ is constant on the set of all rank-one idempotents of $\ca{A}_\rho$ supported in $C \times C$, and will henceforth be denoted by $\omega_C$. 
			
			\smallskip
			
			\noindent  Now we prove \eqref {eq:homogeneity}. First of all, if $X$ is an idempotent, then \eqref{eq:homogeneity} follows from the fact that  $X$ can be decomposed as a sum of mutually orthogonal rank-one idempotents supported in $C \times C$ (Lemma \ref{le:idempotents sum of rank one}) and then apply the orthoadditivity of $\phi$ on $\idem(\ca{A}_\rho)$. For general $X$, let $P_C \in \idem(\ca{A_\rho})$ be the corresponding central idempotent (defined by \eqref{eq:P_S}). We have
			\begin{align*}
				\phi(\lambda X) &= \phi(X \bullet (\lambda P_C)) = \phi(X) \bullet \phi(\lambda P_C) = \omega_C(\lambda) \phi(X) \bullet \phi(P_C) \\
				&= \omega_C(\lambda) \phi(X \bullet P_C) = \omega_C(\lambda) \phi(X),
			\end{align*}
			which implies \eqref{eq:homogeneity}. The injectivity of $\phi$ clearly implies the injectivity of $\omega_C$.
		\end{claimproof}
		
		\begin{claim}\label{cl:h is multiplicative}
			For each central class $C \in \ca{Q}$, the map $\omega_C$ is multiplicative.
		\end{claim}
		\begin{claimproof}
			Fix $C \in \ca{Q}$, an arbitrary $i \in C$ and $\lambda, \mu \in \C$. By Claim \ref{cl:h exists} we have
			\begin{align*}
				\omega_C(\lambda\mu)E_{ii} &= \phi((\lambda\mu) E_{ii}) = \phi((\lambda E_{ii}) \bullet (\mu E_{ii})) = \phi(\lambda E_{ii}) \bullet \phi(\mu E_{ii})\\
				&= \omega_C(\lambda)\omega_C(\mu)E_{ii},
			\end{align*}
			which implies $\omega_C(\lambda\mu) = \omega_C(\lambda)\omega_C(\mu)$, as desired.
		\end{claimproof}
		
		\begin{claim}\label{cl:matrix units}
			Fix a central class  $C \in \ca{Q}$. Then
			$$\phi(E_{ij}) \propto E_{ij}, \quad  \text{ for all } (i,j) \in \rho \cap (C \times C)$$
			(always the case when $\phi$ is assumed to be multiplicative) or
			$$\phi(E_{ij}) \propto E_{ji}, \quad  \text{ for all } (i,j) \in \rho \cap (C \times C).$$
		\end{claim}
		\begin{claimproof}
			Fix some $(i,j) \in \rho^\times \cap (C \times C)$. By Claim \ref{cl:preserves support}, we have $\supp \phi(E_{ij}) \subseteq \{i,j\} \times \{i,j\}$ so denote
			$$\phi (E_{ij}) = \sum_{(r,s) \in \{i,j\} \times \{i,j\}}\alpha_{rs} E_{rs}, \quad \alpha_{rs} \in \C$$
			On one hand we have
			\begin{equation}\label{eq:Eij nilpotent}
				0 = \phi(E_{ij} \bullet E_{ij}) = \phi(E_{ij}) \bullet \phi(E_{ij}) = \phi(E_{ij})^2,
			\end{equation}
			and on the other
			\begin{align*}
				\omega_C\left(\frac12\right)\phi(E_{ij}) \quad &\leftstackrel{\text{Claim } \ref{cl:h exists}}= \phi\left(\frac12 E_{ij}\right) = \phi(E_{ii} \circ E_{ij})\stackrel{\eqref{eq:identity on diagonal idempotents}}= E_{ii}\circ \phi(E_{ij}) \\
				&=\frac12 \alpha_{ij}E_{ij} +\frac12 \alpha_{ji}E_{ji} + \alpha_{ii}E_{ii}
			\end{align*} (if $\phi$ preserves $\circ$)
			or
			\begin{align*}
				\phi\left( E_{ij}\right) = \phi(E_{ii}E_{ij})\stackrel{\eqref{eq:identity on diagonal idempotents}}= E_{ii}\phi(E_{ij}) =\alpha_{ij}E_{ij} + \alpha_{ii}E_{ii}
			\end{align*}
			(if $\phi$ is multiplicative).
			In either case, by \eqref{eq:Eij nilpotent} it ultimately follows $\alpha_{ii} = \alpha_{jj} = 0$ and (by injectivity)
			$$\phi(E_{ij}) \propto \begin{cases}
				E_{ij} \text{ or } E_{ji}, &\quad\text{ if }\bullet = \circ \\
				E_{ij}, &\quad\text{ if } \bullet = \cdot.
			\end{cases}$$
			It remains to prove that in the case $\bullet = \circ$, the same option happens for each pair in $\rho \cap (C \times C)$. For the sake of concreteness, assume that $\phi(E_{ij}) \propto E_{ij}$. If $(j,i) \in \rho$, then clearly $\phi(E_{ji}) \propto E_{ji}$ as otherwise
			$$\phi\left( \frac12 \left(E_{ii}+E_{jj}\right)\right) = \phi(E_{ij} \circ E_{ji}) = \phi(E_{ij}) \circ \phi(E_{ji}) \propto E_{ij} \circ E_{ij} = 0$$
			is a contradiction. Next we show that $\phi(E_{ik}) \propto E_{ik}$ for any $k \in (\rho^\times)(i)$, and $\phi(E_{lj}) \propto E_{lj}$ for any $l \in (\rho^\times)^{-1}(j)$.
			\begin{itemize}
				\item[--] Suppose that $k \in (\rho^\times)(i) \setminus \{j\}$ and that $\phi(E_{ik}) \propto E_{ki}$. Then 
				$$0 = \phi(E_{ij} \circ E_{ik}) = \phi(E_{ij}) \circ \phi(E_{ik}) \propto E_{ij} \circ E_{ki} = \frac12 E_{kj}$$
				is a contradiction so it must be $\phi(E_{ik})  \propto E_{ik}$.
				\item[--] Suppose that $l \in (\rho^\times)^{-1}(j)\setminus \{i\}$ and that $\phi(E_{lj}) \propto E_{jl}$. Then 
				$$0 = \phi(E_{ij} \circ E_{lj}) = \phi(E_{ij}) \circ \phi(E_{lj}) \propto E_{ij} \circ E_{jl} = \frac12 E_{il}$$
				is a contradiction so it must be $\phi(E_{lj})  \propto E_{lj}$.
			\end{itemize}
			Now we can apply Lemma \ref{le:contains entire CxC} to the set
			$$\ca{S} := \{(r,s)\in \rho^\times : \phi(E_{rs}) \propto E_{rs}\}.$$ 
			Since $(i,j) \in \ca{S}$, the set $\ca{S}$ contains all $(r,s) \in \rho^\times \cap (C \times C)$, which proves the claim.
		\end{claimproof}
		
		In view of Claim \ref{cl:matrix units}, we can define a map $g : \rho \to \C^\times$ which to a pair $(i,j) \in \rho$ assigns a scalar $g(i,j)$ such that
		$$\phi(E_{ij}) = g(i,j)E_{ij} \qquad \text{ or }\qquad \phi(E_{ij}) = g(i,j)E_{ji}$$
		(we know that $g(i,j) = 1$ if $i = j$, and otherwise exactly one option is true).
		
		\begin{claim}
			For each central class $C \in \ca{Q}$, the map $\omega_C$ is additive.
		\end{claim}
		\begin{claimproof}
			Let $C \in \ca{Q}$. By invoking Claim \ref{cl:matrix units}, without losing generality we can assume that $$\phi(E_{ij}) = g(i,j)E_{ij}, \quad  \text{ for all } (i,j) \in \rho \cap (C \times C).$$
			Since $\abs{C} \ge 2$, there exists some $(i,j) \in \rho^\times \cap (C \times C)$. For fixed $x,y \in \C$ consider the idempotents
			$$E_{ii} + x E_{ij}, E_{jj} + y E_{ij} \in \idem(\ca{A}_\rho).$$
			By Claim \ref{cl:preserves support} we have
			$$\supp \phi(E_{ii} + x E_{ij}) \subseteq \{i,j\} \times \{i,j\}.$$
			Denote $$\phi (E_{ii} + x E_{ij}) = \sum_{(r,s) \in \{i,j\} \times \{i,j\}}\alpha_{rs} E_{rs}, \quad \alpha_{rs} \in \C.$$
			Suppose that $\phi$ preserves $\circ$. We have
			\begin{align*}
				\omega_C\left(\frac12 x\right) g(i,j) E_{ij} \quad &\leftstackrel{\text{Claim }\ref{cl:h exists}}= \phi\left(\frac12 x E_{ij}\right) = \phi((E_{ii} + x E_{ij}) \circ E_{jj})\stackrel{\eqref{eq:identity on diagonal idempotents}}= \phi(E_{ii} + x E_{ij}) \circ E_{jj} \\
				&= 
				\frac12 \alpha_{ij}E_{ij} + \frac12 \alpha_{ji}E_{ji} + \alpha_{jj} E_{jj}.
			\end{align*}
			Since $\phi(E_{ii} + x E_{ij})$ is an idempotent and $\omega_C^{-1}(\{0\}) = \{0\}$, we conclude
			$$\alpha_{ij} =
			2 \omega_C\left(\frac12 x\right)  g(i,j), \qquad \alpha_{ji} = \alpha_{jj} = 0, \qquad \alpha_{ii} = 1.$$
			Hence
			$$\phi(E_{ii} + x E_{ij}) = E_{ii} + 2 \omega_C\left(\frac12 x\right)  g(i,j) E_{ij}.$$
			In an analogous way we arrive at the equality
			$$\phi(E_{jj} + y E_{ij}) = 
			E_{jj} + 2 \omega_C\left(\frac12 y\right)  g(i,j) E_{ij}.$$
			We have
			\begin{align*}
				\omega_C\left(\frac{x+y}2\right) g(i,j) E_{ij}\quad &\leftstackrel{\text{Claim }\ref{cl:h exists}}= \phi\left(\frac{x+y}2 E_{ij}\right) = \phi((E_{ii} + x E_{ij}) \circ (E_{jj} + y E_{ij})) \\
				&= \phi(E_{ii} + x E_{ij}) \circ \phi(E_{jj} + y E_{ij}) \\
				&= \left(E_{ii} + 2 \omega_C\left(\frac12 x\right) g(i,j) E_{ij}\right)  \circ \left(E_{jj} + 2 \omega_C\left(\frac12 y\right) g(i,j) E_{ij}\right)\\
				&= \left(\omega_C\left(\frac12 x\right) + \omega_C\left(\frac12 y\right)\right)g(i,j) E_{ij}
			\end{align*}
			and hence
			$$\omega_C\left(\frac{x+y}2\right) = \omega_C\left(\frac12 x\right) + \omega_C\left(\frac12 y\right).$$
			As $x,y \in \C$ were arbitrarily chosen, this closes the proof for $\circ$. 
			
			\smallskip
			
			\noindent   If $\phi$ is a multiplicative map, a similar calculation implies
			$$\phi(E_{ii} + x E_{ij}) = 
			E_{ii} + \omega_C\left(x\right)  g(i,j) E_{ij}, \qquad \phi(E_{jj} + y E_{ij}) = E_{jj} + \omega_C\left(y\right)  g(i,j) E_{ij}$$
			and hence
			\begin{align*}
				\omega_C(x+y) g(i,j)E_{ij} &= \phi((x+y)E_{ij}) = \phi((E_{ii}+xE_{ij})(E_{jj}+yE_{ij})) \\
				&= (\omega_C(x)+\omega_C(y))g(i,j)E_{ij}
			\end{align*}
			which likewise implies the desired claim.
		\end{claimproof}
		It follows that each map $\omega_C : \C \to \C$ is an injective ring endomorphism of $\C$ (i.e.\ a monomorphism), and hence acts as the identity on the subfield $\Q$ of rational numbers.
		
		\begin{claim}\label{cl:Q-linear}
			$\phi$ is a $\Q$-homogeneous map.
		\end{claim}
		\begin{claimproof}
			First of all, for $\lambda \in \Q$ we have 
			\begin{align*}
				\phi(\lambda I)&= \phi\left(\sum_{C \in \ca{Q}} \lambda P_C \right)\stackrel{\text{Lemma } \ref{le:basic properties} \text{ (g)}}= \sum_{C \in \ca{Q}} \phi(\lambda P_C)  \stackrel{\text{Claim } \ref{cl:h exists}}=\sum_{C \in \ca{Q}} \omega_C(\lambda) \phi(P_C)\\
				& \leftstackrel{\eqref{eq:identity on diagonal idempotents}} = \sum_{C \in \ca{Q}} \lambda P_C = \lambda I.
			\end{align*}
			Now, for arbitrary $X \in \ca{A}_\rho$ and $\lambda \in \Q$ we have
			$$\phi(\lambda X) = \phi(X \bullet (\lambda I)) = \phi(X) \bullet \phi(\lambda I) = \lambda \phi(X).$$
		\end{claimproof}
		
		\begin{claim}
			The map $g$ is transitive.
		\end{claim}
		\begin{claimproof}
			Fix $(i,j), (j,k) \in \rho$. Then $(i,k) \in \rho$ as well. Since $i,j,k \in C$ for some central class $C\in \ca{Q}$, for concreteness assume that
			$$\phi(E_{ij}) = g(i,j)E_{ij}, \qquad \phi(E_{jk}) = g(j,k)E_{jk}, \qquad \phi(E_{ik}) = g(i,k)E_{ik}.$$
			First assume that $\phi$ preserves $\circ$. If $i \ne k$, then
			\begin{align*}
				\frac12 g(i,k)E_{ik} &\stackrel{\text{Claim }\ref{cl:Q-linear}}= \phi\left(\frac12 E_{ik}\right) =\phi(E_{ij} \circ E_{jk}) = \phi(E_{ij}) \circ \phi(E_{jk}) \\
				&= \frac12 g(i,j)g(j,k)E_{ik},
			\end{align*}
			which implies $g(i,k) = g(i,j)g(j,k)$.
			Similarly, if $i = k$, then
			\begin{align*}
				\frac{1}{2}\left(E_{ii} + E_{jj}\right) \qquad &\leftstackrel{\text{Claim }\ref{cl:Q-linear}, \eqref{eq:identity on diagonal idempotents}}= \phi\left(\frac{1}{2}\left(E_{ii} + E_{jj}\right)\right) =\phi(E_{ij} \circ E_{ji}) = \phi(E_{ij}) \circ \phi(E_{ji}) \\
				&= \frac12 g(i,j)g(j,i)(E_{ii} + E_{jj})
			\end{align*}
			which implies $g(i,i) = 1 = g(i,j)g(j,i)$. The proof is even shorter for multiplicative maps.
		\end{claimproof}
		
		Therefore, by passing to the map $(g^*)^{-1} \circ \phi$, without loss of generality we can assume that for each $C \in \ca{Q}$ there exists an assignment $\dagger_C \in \{\text{identity}, \text{transposition}\}$  (always the identity when $\phi$ is multiplicative) so that
		$$\phi(E_{ij}) = E_{ij}^{\dagger_C}, \quad  \text{ for all } (i,j) \in \rho\cap (C \times C).$$
		
		\begin{claim}\label{cl:Jordan triple product}
			Let $X \in \ca{A}_\rho$ and $P \in \idem(\ca{A_\rho})$. Then
			$$\phi(PXP) = \phi(P)\phi(X)\phi(P).$$
		\end{claim}
		\begin{claimproof}
			This is clearly true for multiplicative maps, so assume that $\phi$ is $\circ$-preserving. One easily verifies the equality
			\begin{equation}\label{eq:P-Pperp}
				(P-P^\perp) \circ (X \circ P) = PXP.
			\end{equation}
			We also have
			$$\phi(P-P^\perp) \stackrel{\text{Lemma } \ref{le:basic properties}}= \phi(P) + \phi(-P^\perp)  \stackrel{\text{Claim }\ref{cl:Q-linear} \text{ and Lemma } \ref{le:basic properties}}= \phi(P) - \phi(P)^\perp.$$
			Hence
			\begin{align*}
				\phi(PXP)\,\, &\leftstackrel{\eqref{eq:P-Pperp}}= \phi(  (P-P^\perp) \circ (X \circ P)) = (\phi(P)-\phi(P)^\perp) \circ (\phi(X) \circ \phi(P)) \\
				&\leftstackrel{\eqref{eq:P-Pperp}}= \phi(P)\phi(X)\phi(P).
			\end{align*}
		\end{claimproof}

		\begin{claim}\label{cl:supported in CxC}
			For all $C \in \ca{Q}$ and $X \in \ca{A}_\rho$ with $\supp X \subseteq C \times C$ we have
			$$\phi(X) = \sum_{(i,j)\in\rho \cap (C \times C)} \omega_C(X_{ij})E_{ij}^{\dagger_C} = \omega_C(X)^{\dagger_C}.$$
		\end{claim}
		\begin{claimproof}
			We prove the claim for $\bullet = \circ$, as the multiplicative case is similar, only simpler. Fix $C \in \ca{Q}$. For concreteness, assume that $\dagger_C = \id$ and fix some $X \in \ca{A}_\rho$ such that $\supp X \subseteq C \times C$. Clearly, by Claim \ref{cl:preserves support}, $\phi(X)$ is also supported in $C \times C$. Let $(i,j) \in C\times C$. If $i = j$, we have
			\begin{align*}
				\omega_C(X_{ii})E_{ii} = \phi(X_{ii}E_{ii}) = \phi(E_{ii}XE_{ii}) \stackrel{\text{Claim }\ref{cl:Jordan triple product}}= E_{ii}\phi(X)E_{ii} = \phi(X)_{ii}E_{ii},
			\end{align*}
			so $\phi(X)_{ii} = \omega_C(X_{ii})$. Now assume $i \ne j$. Assume first that $(i,j),(j,i) \in \rho$. As $\omega_C$ is multiplicative and acts as the identity on $\Q$, it follows
			\begin{align*}
				\frac12 \omega_C(X_{ij})E_{ji} &= \phi\left(\frac12 X_{ij} E_{ji}\right) = \phi\left(\frac12 E_{ji}XE_{ji}\right) = \phi((E_{ji} \circ X ) \circ E_{ji})\\
				&= (\phi(E_{ji}) \circ \phi(X)) \circ \phi(E_{ji}) = (E_{ji} \circ \phi(X)) \circ E_{ji}\\
				&= \frac12 \phi(X)_{ij}E_{ji},
			\end{align*}
			which implies $\phi(X)_{ij} = \omega_C(X_{ij})$.
			Suppose now that $(i,j) \in \rho$ but $(j,i) \notin \rho$ (so that $X_{ji}= 0$). We have
			\begin{align*}
				\frac14 \omega_C(X_{ij})E_{ij} &= \phi\left(\frac14 X_{ij} E_{ij}\right) = \phi\left(\frac14 X_{ij} E_{ij} + \frac14 X_{ji}E_{ji}\right) \\
				&= \phi\left(\frac14(E_{ii}XE_{jj} + E_{jj}XE_{ii})\right) = \phi((E_{ii} \circ X) \circ E_{jj}) \\
				&= (E_{ii} \circ \phi(X) )\circ E_{jj} = \frac14 \phi(X)_{ij}E_{ij} + \frac14 \phi(X)_{ji}E_{ji},
			\end{align*}
			so $\phi(X)_{ij} = \omega_C(X_{ij})$ and $\phi(X)_{ji} = 0$. Finally, the same calculation also shows that $\phi(X)_{ij} = \phi(X)_{ji} = 0$ for each $i,j \in C$ such that $(i,j),(j,i) \notin \rho$. This proves the claim.
		\end{claimproof}
		
		We are now in the position to finish the proof of the theorem.
		\begin{claim}
			For each $X \in \ca{A}_\rho$, we have
			$$\phi(X) = \sum_{C \in \ca{Q}} \omega_C(P_C X)^{\dagger_C}.$$
		\end{claim}
		\begin{claimproof}
			Indeed,
			\begin{align*}
				\phi(X) &= \phi(X) \bullet I = \phi(X) \bullet \left(\sum_{C \in \ca{Q}}P_C\right) = \sum_{C \in \ca{Q}} (\phi(X) \bullet P_C) \\
				&\leftstackrel{\eqref{eq:identity on diagonal idempotents}}= \sum_{C \in \ca{Q}} (\phi(X) \bullet \phi(P_C)) = \sum_{C \in \ca{Q}} \phi(X \bullet P_C) = \sum_{C \in \ca{Q}} \phi(\underbrace{P_C X}_{\text{supported in $C \times C$}})\\
				&\leftstackrel{\text{Claim }\ref{cl:supported in CxC}}= \sum_{C \in \ca{Q}}  \omega_C(P_C X)^{\dagger_C}.
			\end{align*}
		\end{claimproof}
	\end{proof}
	
	\begin{remark}
		Using the setting of Theorem \ref{thm:main result}, adding the additional assumption of the continuity at a single point for the map $\phi$ ensures that all monomorphisms $\omega_C : \C \to \C$ are either the identity or the complex conjugation. However, in the absence of this continuity assumption there exist infinitely many such monomorphisms (even automorphisms) $\omega_C$.
	\end{remark}
	
	\begin{remark}
		In contrast to \cite[Theorem~4.9]{GogicTomasevic}, the injectivity assumption of the map $\phi$ in Theorem \ref{thm:main result} cannot be relaxed to the condition that only $\phi(E_{ij}) \ne 0$ for all $(i,j) \in \rho$. This is illustrated by a constant map that assigns each matrix to a fixed nonzero idempotent.
	\end{remark}
	
	The next example demonstrates that Theorem \ref{thm:main result} cannot be generalized to arbitrary unital subalgebras of $M_n$. 
	\begin{example}\label{ex:nonSMAexample}
		Consider $\ca{A}\subseteq M_5$ defined by
		\begin{equation*}
			\ca{A}:=\left\{\begin{bmatrix} x_{11} & 0 & 0 & 0 & 0\\
				x_{21} & y & z & 0 & 0 \\
				0 & 0 & x_{33} & 0 & 0 \\
				0 & 0 & z & y & x_{45} \\
				0 & 0 & 0 & 0 & x_{55}  \end{bmatrix}  : x_{ij}, y,z \in \C \right\}.
		\end{equation*}
		One can easily check that $\ca{A}$ is a central subalgebra of $M_5$. On the other hand (as in the proof of (ii) $\implies$ (i) of Theorem \ref{thm:main result}), choose any injective multiplicative non-additive function $\omega : \C \to \C$ and define a map $\phi : \ca{A}\to M_5$ by
		$$\phi\left(\begin{bmatrix} x_{11} & 0 & 0 & 0 & 0\\
			x_{21} & y & z & 0 & 0 \\
			0 & 0 & x_{33} & 0 & 0 \\
			0 & 0 & z & y & x_{45} \\
			0 & 0 & 0 & 0 & x_{55}  \end{bmatrix}\right) := \begin{bmatrix} x_{11} & 0 & 0 & 0 & 0 \\
			x_{21} & y & z & x_{45} & 0 \\
			0 & 0 & x_{33} & 0 & 0 \\
			0 & 0 & 0 & x_{55} & 0 \\
			0 & 0 & 0 & 0 & \omega(x_{11})  \end{bmatrix}.$$
		It is then straightforward to verify that $\phi$ is an injective non-additive map that is both multiplicative and Jordan multiplicative.
	\end{example}
	
	We conclude the paper with a brief discussion of a  generalization of Theorem~\ref{thm:main result} to the class of SMAs over more general fields.
	\begin{remark}\label{rem:extension to fields}
		Let $\F$ be an arbitrary field and let  $\mathcal{A}_\rho \subseteq M_n(\F)$ be an SMA. By examining the proof of \cite[Theorem 3.4]{GogicTomasevic}, one can observe that for any commuting family $\mathcal{F} \subseteq \idem(\ca{A}_\rho)$ of idempotents in $\mathcal{A}_\rho$ there exists an invertible matrix $S \in \mathcal{A}_\rho^\times$ such that $S \mathcal{F} S^{-1} \subseteq \mathcal{D}_n(\F)$.
		
		\smallskip 
		
		In particular, if the characteristic of $\F$ is not $2$, then combining the above observation with the remainder of the proof of Theorem~\ref{thm:main result} yields that the implications \textnormal{(i)} $\implies$ \textnormal{(iii)} $\implies$ \textnormal{(ii)} continue to hold for arbitrary SMAs $\mathcal{A}_\rho \subseteq M_n(\mathbb{F})$. Furthermore, the implication \textnormal{(ii)} $\implies$ \textnormal{(i)} also holds whenever $\mathbb{F}$ admits a non-additive, multiplicative, injective self-map $\omega$. For example, besides the case $\mathbb{F} = \mathbb{C}$, this condition is satisfied for $\mathbb{F} = \mathbb{R}$ or any of its subfields (one can choose $\omega(x)=x^3$).
	\end{remark}

\end{document}